\documentclass{amsart}
\usepackage{amsfonts,amssymb,amscd,amsmath,enumerate,verbatim,calc}
\usepackage[all]{xy}

\newcommand{\CM}{Cohen-Macaulay}

\newcommand{\wrt}{with respect to}

\newcommand{\n}{\mathfrak{n} }
\newcommand{\m}{\mathfrak{m} }
\newcommand{\eU}{\mathfrak{U} }
\newcommand{\eC}{\mathfrak{C} }
\newcommand{\q}{\mathfrak{q} }
\newcommand{\g}{\mathfrak{g} }

\newcommand{\ZZ}{\mathbb{Z} }

\newcommand{\FF}{\mathbb{F}}

\newcommand{\C}{\mathbf{C} }

\newcommand{\rt}{\rightarrow}

\newcommand{\ov}{\overline}

\newcommand{\image}{\operatorname{image}}
\newcommand{\Gr}{\operatorname{Gr}}
\newcommand{\Ass}{\operatorname{Ass}}

\newcommand{\mode}{\operatorname{mod} }

\newcommand{\rk}{\operatorname{rk}}
\newcommand{\height}{\operatorname{height}}
\newcommand{\charp}{\operatorname{char}}
\newcommand{\Der}{\operatorname{Der}}
\newcommand{\rank}{\operatorname{rank}}
\newcommand{\Proj}{\operatorname{Proj}}
\newcommand{\Sing}{\operatorname{Sing}}

\newcommand{\Hom}{\operatorname{Hom}}

\theoremstyle{plain}

\newtheorem{theorem}{Theorem}[section]

\newtheorem{lemma}[theorem]{Lemma}
\newtheorem{proposition}[theorem]{Proposition}
\newtheorem{question}[theorem]{Question}

\theoremstyle{definition}

\newtheorem{s}[theorem]{}
\newtheorem{remark}[theorem]{Remark}
\newtheorem{example}[theorem]{Example}

\theoremstyle{remark}

\begin{document}

\title{On the ring of differential operators of certain regular domains }
 \author{Tony J. Puthenpurakal}
\date{\today}
\address{Department of Mathematics, Indian Institute of Technology Bombay, Powai, Mumbai 400 076, India}
\email{tputhen@math.iitb.ac.in}
\begin{abstract}
Let $(A,\m)$ be a complete equicharacteristic Noetherian domain of dimension $d + 1 \geq 2$. Assume $k = A/\m$ has characteristic zero and that $A$ is not a regular local ring. Let $\Sing(A)$  the singular locus of $A$ be defined by an ideal $J$ in $A$. Note $J \neq 0$. Let $ f \in J$ with $f \neq 0$. Set $R = A_f$. Then $R$ is a regular domain of dimension $d$.  We show $R$ contains naturally a field $\ell \cong k((X))$.  Let $\g$ be the set of $\ell$-linear derivations of $R$ and let $D(R)$ be the subring of $\Hom_\ell(R,R)$ generated by $\g$ and the multiplication operators defined by elements in the ring $R$. We show that $D(R)$, the ring of $\ell$-linear differential operators on $R$,  is a left, right Noetherian ring of global dimension $d$.  This enables us to prove Lyubeznik's conjecture on $R$ modulo a conjecture on roots of Bernstein-Sato polynomials over power series rings.
\end{abstract}

\maketitle

\section{introduction}
Let $K$ be a field of characteristic zero  and let $R$ be a commutative Noetherian domain containing $K$ as a subring. Let $\g$ be the set of $K$-linear derivations of $R$ and let $D(R)$ be the subring of $\Hom_K(R,R)$ generated by $\g$ and the multiplication operators defined by elements in the ring $R$.
The ring $D(R)$ is called the ring of \emph{$K$-linear differential operators} on $R$.  In general $D(R)$ does not have good properties. However in the following cases it is known that $D(R)$ is both left and right Noetherian with finite global dimension:
\begin{enumerate}
\item
$R = K[X_1,\ldots, X_n]$. In this case $D(R) = A_n(K)$ the $n^{th}$-Weyl algebra over $K$.  We have global dimension of $D(R)$ is equal to $n$, see \cite[Chapter 2, Theorem 3.15 ]{B}. 
\item
$R = K[[X_1,\ldots, X_n]]$. In this case  global dimension of $D(R)$ is equal to $n$, see \cite[Chapter 3, Proposition 1.8]{B}.
\item
Let $K = \mathbb{C}$ and let $V$ be a non-singular affine $K$-variety. Let $R$ be the co-ordinate ring of $V$. In this case  global dimension of $D(R)$ is equal to $\dim V$, see \cite[Chapter 3, Theorem 2.5]{B}.
\item
Let $R = \mathbb{C}\{z_1,\ldots,z_n\}$ be the local ring of convergent power series in $n$-variables.  In this case  global dimension of $D(R)$ is equal to $n$, see \cite[p.\ 197]{B}.
\end{enumerate}
In this paper we describe a \emph{new}  vast class of Noetherian domains $R$ with $D(R)$ both left and right Noetherian and with finite global dimension. 

\s \label{Hypothesis} \textit{Setup:} 
Let $(A,\m)$ be a complete equicharacteristic Noetherian domain of dimension $d + 1 \geq 2$. Assume $k = A/\m$ has characteristic zero and that $A$ is not a regular local ring. Let  $\Sing(A)$  the singular locus of $A$ be defined by an ideal $J$ in $A$. Note $J \neq 0$. Let $ f \in J$ with $f \neq 0$. Set $R = A_f$. Then $R$ is a regular domain of dimension $d$.  In \ref{subfield}  we show that  $R$ contains naturally a field $\ell \cong k((X))$.  Let $\g$ be the set of $\ell$-linear derivations of $R$ and let $D(R)$ be the subring of $\Hom_\ell(R,R)$ generated by $\g$ and the multiplication operators defined by elements in the ring $R$.  The main result of this paper is
\begin{theorem} \label{main}[ with hypotheses as in \ref{Hypothesis}]
The ring  $D(R)$  is a left, right Noetherian ring of global dimension $d$.  
\end{theorem}

\s \label{app} \textit{Application:} Lyubeznik conjectured, see \cite{Lyu-3},  that if $S$ is a regular ring and $I$ is an ideal in $S$ then for any $i \geq 0$ the set $\Ass_S H^i_I(S)$ is finite. This conjecture is known to be true in the following cases:
\begin{enumerate}
\item
$S$ contains a field $K$  with $\charp K = p > 0$, see \cite{HuSh}.
\item
$S$ is local  and containing a field $K$ with $\charp K = 0$, see \cite{Lyu-1}.
\item
$S$ is a regular affine $K$-algebra (here $\charp K = 0$), see \cite{Lyu-1}.
\item
$S$ is an unramifed regular local ring, see \cite{Lyu-4}.
\item
$S$ is a smooth $\mathbb{Z}$-algebra, see \cite{BB}. 
\end{enumerate}
However none of the techniques used to prove the above results can be used to verify Lyubeznik's conjecture for rings $R$ as given in \ref{Hypothesis}.  We show that Theorem \ref{main} and an affirmative answer to a question regarding Bernstein-Sato polynomials 
of a formal power series (see \ref{bs} and \ref{bs-K})  enables us  to verify Lyubeznik's conjecture in this case.

Here is an overview of the contents of the paper. In section two we discuss some preliminaries that we need. In section three we discuss our result on ranks of certain modules of derivations. The main result in this section is Theorem \ref{chief}. We prove Theorem \ref{main} in section four. In the next section we give an application of our result to Lyubeznik's conjecture. Finally in section we give bountiful number of examples of regular rings satisfying our hypothesis \ref{Hypothesis}. 

\section{Some preliminaries}
Let $(A,\m)$ be a complete equicharacteristic Noetherian domain of dimension $d + 1 \geq 2$. Assume $k = A/\m$ has characteristic zero and that $A$ is not a regular local ring. Let  $\Sing(A)$  the singular locus of $A$ be defined by an ideal $J$ in $A$. Note $J \neq 0$. Let $ f \in J$ with $f \neq 0$. Set $R = A_f$. Then $R$ is a regular domain of dimension $d$.  In this section we prove some preliminary facts about $R$ and $A$. We note that $A$ contains a field isomorphic to $k$. For convenience we also denote it with $k$.

We first prove
\begin{proposition}\label{max}[ with hypotheses as above.]
Let $\n$ be a maximal ideal in $R$. Then $\n = \q R$ where $\q$ is a prime ideal of height $d$ in $A$. In particular $\dim R = d$. 
\end{proposition}
\begin{proof}
Note $f \notin \q$. Suppose if possible $\height \q \leq d-1$. As $A$ is complete it is catenary. So $\height(\m / \q) \geq 2$. In particular $\dim A/\q \geq 2$. The image of $f$ is non-zero in $A/\q$.  It is elementary to see that $A/\q$ has infinitely many prime ideals of height one. We can choose one, say $\ov{P} = P/\q$ not containing $\ov{f}$. Thus $P$ is a prime ideal in $A$ not containing $f$ and $P$ strictly contains $\q$. It follows that $\n =  \q R$ is
not a maximal ideal of $R$, a contradiction.
\end{proof}

\s \label{subfield} Consider the map $\phi \colon k[[X]] \rt A$ which maps $k$ identically to $k$ and $X$ to $f$. As $A$ is a domain it is clear that $\phi $ is an injective map.
Inverting $X$ we get a map $\psi \colon k((X)) \rt A_X$. It is clear that $A_X = A_f = R$. Thus $R$ naturally contains a field $\ell \cong k((X))$. We also note that
$\image \phi = k[[f]]$ and $\ell = k((f))$.

The following is a crucial ingredient to prove Theorem \ref{main}.
\begin{lemma}\label{field-extn}
Let $\n$ be a maximal ideal of $R$. Then $R/\n$ is a finite extension of $\ell$.
\end{lemma}
\begin{proof}
By Proposition \ref{max} we get that $\n = \q R$ where $\q$ is a prime ideal of height $d$ in $A$ not containing $f$. The map $\phi \colon k[[X]] \rt A$  as in \ref{subfield}
descends to a map $\ov{\phi} \colon k[[X]] \rt A/\q$.  Set $T = A/\q$ and $S = k[[X]]$.  As $(\q , f)$ is $\m$-primary in $A$ we get that $T/XT = A/(\q, f)$ is a finite dimensional $k$-vector space. We also get 
\[
\bigcap_{n\geq 1} X^n T \subseteq \bigcap_{n\geq 1} f^n T  \subseteq \bigcap_{n \geq 1} \m^n T = 0.
\] 
Thus $T$ is seperated \wrt \ $(X)$-topology of $S$. It follows that $T$ is a finite $S$-module, see  \cite[Theorem 8.4]{Mat}. Therefore the quotient field of $A/\q$ will be a finite extension of quotient field of $S$. The  result follows.
\end{proof}
We will use the next result in the next section.
\begin{lemma} \label{choose-elts}[ with hypotheses as above:]
Let  $\q$ be a prime of height $d$ in $A$ such that $f \notin \q$.  Let $\kappa(\q)$ be the residue field of $A_\q$. Then there exists $y_1,\ldots, y_d \in \q$ such that
\begin{enumerate}[\rm (1)]
\item
$\height (f,y_1,\ldots, y_j)  = j+1$ for $j = 0, \ldots, d$.
\item 
The images of $y_1,\ldots, y_j$ in the $\kappa(\q)$-vector space $\q A_\q / \q^2 A_\q$ is linearly independent for $j = 1,\ldots, d$.
\item
$f, y_1,\ldots, y_d$ is a system of parameters of $A$.
\item
$(y_1,\ldots, y_d)A_\q  = \q A_\q$.
\end{enumerate}
\end{lemma}
\begin{proof}
(1)  and (2). As $A$ is a domain we get that $\height (f) = 1$. Now
suppose $y_1,\ldots, y_j$ is already chosen where $0 \leq j < d$. We choose $y_{j+1}$ as follows: \\
(a) Let $P_1,\ldots, P_s$ be all the minimal  primes of $(f,y_1,\ldots,y_j)$ of height $j+1$.  We claim that $\q \nsubseteq P_i$ for all $i = 1,\ldots, s$. We have to consider two cases: \\
case (i) : $j \leq d-2$. Then as $\height P_i < d$  for all $i$, we get the result. \\
case(ii): $j = d -1$. If $\q \subseteq P_i$ for some $i$ then as both these prime ideals have height $d$ we get $\q = P_i$.  We then get $f \in \q$, a contradiction. 

(b) Set
$$J  = \left((y_1,\ldots, y_j)A_\q + \q^2A_\q\right)\cap A. $$
Then $J \subseteq \q$. We claim that $\q \nsubseteq J$. If this is so we get $\q = J$ and therefore
\[
\q A_\q = (y_1,\ldots, y_j)A_\q + \q^2 A_\q \quad \text{and so by Nakayama's Lemma} \  \q A_\q = (y_1,\ldots, y_j)A_\q.
\]
This implies that $\dim A_\q \leq j <d$, a contradiction.

By prime avoidance there exists
\[
y_{j+1} \in \q \setminus J \cup \left( \cup_{i =1}^{s}P_i\right).
\]
Then note that $y_1,\ldots,y_{j+1}$ satisfies the conditions of (1) and (2).

(3) This follows since by (1) we have $\height (f, y_1, \ldots, y_d) = d+1 = \dim A$.

(4) As $A_\q$ is a regular local ring of dimension $d$ we get that $\q A_\q/\q^2 A_\q$ is a $d$-dimensional $\kappa(\q)$-vector space. The result follows from (2).
\end{proof}
\section{ranks of modules of derivations}
Let $T, S$ be  commutative  Noetherian rings.  Assume  $S$ is a $T$-algebra. Let $\Der_T(S)$ denote the set of $T$-linear derivations on $S$. The $S$-module $\Der_T(S)$ need not be finitely generated. However there are many natural instances where it is so.

\s \label{setup} Our setup in this section will be as in \ref{Hypothesis}.  Let us recall it here.  Let $(A,\m)$ be a complete equicharacteristic Noetherian domain of dimension $d + 1 \geq 2$. Assume $k = A/\m$ has characteristic zero and that $A$ is not a regular local ring. Let  $\Sing(A)$  the singular locus of $A$ be defined by an ideal $J$ in $A$. Note $J \neq 0$. Let $ f \in J$ with $f \neq 0$. Set $R = A_f$. Then $R$ is a regular domain of dimension $d$.  We note that $A$ contains a field isomorphic to $k$. For convenience we also denote it with $k$. By \ref{subfield} $R$ contains a field $\ell$ isomorphic to $k((X))$. Furthermore  by \ref{field-extn}  if $\n$ is a maximal ideal of $R$ then the field $R/\n$ is a finite extension of $\ell$. Also if $\n = \q R$ is a maximal ideal of $R$ with $\q$ a prime ideal in $A$ then $\height \q = d$, see \ref{max}.

We first prove:
\begin{proposition}\label{rank-d-A}
[ with hypotheses as in \ref{setup}.]  The $A$-module $\Der_k(A)$ is finitely generated with rank $= d+1$.
\end{proposition}
\begin{proof}
By \cite[Theorem 30.7]{Mat},   $\Der_k(A)$ is a finitely generated $A$-module of rank $\leq d + 1$. 
Let $A  = Q/\q$ where $Q = k[[x_1,\ldots, x_n]] $ and $\q \subseteq (x_1,\ldots, x_n)^2$ is a prime ideal in $Q$.  Let $r = \height \q$. Then $n = d+1+r$. 

Let $T$ be a finitely generated $A$-module. By  equation (6) in Theorem 25.2 of  \cite{Mat} we get an exact sequence of $A$-modules
\[
0 \rt \Der_k(A, T) \rt \Der_k(Q, T) \rt \Hom_A(\q/\q^2, T).
\]
We note that $\Der_k(Q, T) \cong T^n$. Set $T = A$ in the above equation
\[
0 \rt \Der_k(A)  \rt A^n \rt \Hom_A(\q/\q^2, A).
\]
We localize the above equation at $(0)$. We note that $(\q/\q^2)_{(0)}  \cong \q Q_\q/\q^2 Q_\q  \cong \kappa(\q)^r$, here $\kappa(\q)$ is the residue field of $Q_\q$ (this is so as $Q_\q$ is a regular local ring of dimension $r$).  Note $\kappa(\q)$ is also the quotient field of $A$.
So we have an exact sequence
\[
0 \rt \Der_k(A)_{(0)} \rt \kappa(\q)^n \rt \kappa(\q)^r.
\]
Therefore $\rank \Der_k(A) \geq n -r = d+1$. The result follows.
\end{proof}
The following  is the main result of this section:
\begin{theorem}\label{chief}
(with hypotheses as in \ref{setup}.)
Let $T$ be the subring $k[[f]]$ of $A$. Consider $\Der_T(A)$. Then 
\begin{enumerate}[\rm (1)]
\item
$\Der_T(A)$ is a finitely generated $A$-module and 
$\rank \Der_T(A) \geq d$.
\item
$\Der_T(A)_f  = \Der_\ell(R)$. In  particular $\Der_\ell(R)$ is finitely generated as a $R$-module.
\item
Let $\n$ be a maximal ideal of $R$.  Then
\begin{enumerate}[\rm(a)]
\item
$\Der_\ell(R_\n) = \left(\Der_\ell(R)\right)_\n. $
\item
$\Der_\ell(R_\n)$ is free $R_\n$-module of rank $d$.
\end{enumerate}
\item
$\Der_\ell(R)$ is a projective $R$-module of rank $d$.
\end{enumerate}
\end{theorem}
We will need the following two easily proved facts:
\s \label{f1} \emph{Fact 1:} Let $K$ be a field of characteristic zero and let $S = K[[X_1,\ldots, X_n]]$. Let $T$ be an $S$-module, not necessarily finitely generated, such that  $T$ is complete \wrt \   $(X_1,\ldots, X_n)$-adic topology.
Then $\Der_K(S, T) \cong T^n$.

\s \label{f2} \emph{Fact 2:} Let $R \subseteq S$ be an inclusion of Noetherian domains.  Let $I$ be an ideal in $R$ such that $R$ is complete \wrt \ $I$-adic topology. Let $J$ be an ideal in $S$ such that $S$ is complete \wrt \ $J$-adic topology. Assume $IS \subseteq J$. Let $\{r_n \}$ be a convergent sequence in $R$ (in the $I$-adic topology) with $r_n \rt r$. Then $\{ r_n \}$ considered as a sequence in $S$ is convergent in the $J$-adic topology and $\{r_ n\}$ converges to $r$ in $S$. 

We now give
\begin{proof}[Proof of Theorem \ref{chief}]
(1)  Consider the inclusion of rings $k \subseteq T \subseteq A$. By  equation (3) in Theorem 25.1 of \cite{Mat}, for any $A$-module $W$ we have the following exact sequence of $A$-modules 
\begin{equation}\label{c-e-1}
0 \rt  \Der_T(A, W) \rt  \Der_k(A, W) \rt  \Der_k(T, W).
\end{equation}
We now put $W = A$ in (\ref{c-e-1}). Notice that $T \cong k[[X]]$.  As $A$ is complete \wrt \ $\m$-adic topology it is also complete \wrt \ $(f)$-adic topology. 
So $\Der_k(T, A) \cong A$, see \ref{f1}.     
By Proposition \ref{rank-d-A} we get that  $\Der_k(A)$ is finitely generated as an $A$-module and $\rank \Der_k(A) = d +1$. The result follows from (\ref{c-e-1}).

We need some work to prove the remaining assertions:

\textit{Claim-1:}  $\Der_T(A)_f \subseteq \Der_\ell(R)$. In particular $\rank \Der_\ell(R) \geq d$.  \emph{Note we are not yet asserting that $\Der_\ell(R)$ is finitely generated as a $R$-module}.\\
Remark:  Let $L$ be the quotient field of $R$. By the rank of a not-necessarily finitely generated $R$-module $M$, we mean the cardinality of a basis of the $L$-vector space $M\otimes_R L$.\\
It is elementary that $\Der_T(A)_f \subseteq \Der_T(R)$. Now $T = k[[f]]$ and $f$ is invertible in $R$. Let $D \in \Der_T(R)$. We assert that it is $\ell = k((f))$-linear.
To see this let 
\[
v = \frac{1}{f^i}r \quad \text{for some} \ i \geq 1 \ \text{and} \ r \in R.
\]
Then $r = f^i v$. As $D$ is $T$-linear we get $D(r) = f^i D(v)$. It follows that 
\[
D(v) = \frac{1}{f^i} D(r).
\]
Any $\xi \in \ell \setminus T$ is of the form $t/f^i$where $t \in T$ and $i \geq 1$. By the previous argument we get that $D(\xi r) = \xi D(r)$ for any $r \in R$.  Thus $D$ is $\ell$-linear.

Now let $\n$ be a maximal ideal of $R$. Say $\n = \q R$ where $\q$ is a prime ideal in $A$ of height $d$ and $f \notin \q$. By Lemma \ref{choose-elts} there exists  $y_1,\ldots, y_d \in \q$ such that $f, y_1,\ldots, y_d$ is a system of parameters of $A$ and $(y_1,\ldots, y_d)A_\q = \q A_\q$. Set
\[
V = k[[f,y_1, \ldots, y_d]]  = T[[y_1,\ldots, y_d]]. \quad \text{Note} \ V \cong k[[Y_0, Y_1, \ldots, Y_d]].
\]
Also note that $A$ is finitely generated as a $V$-module.

\textit{Claim-2:} $\Der_\ell(R_\n)$ is a free $R_\n$-module of rank $d$. There also exists \\ $\delta_i \in \Der_\ell(R_\n)$ such that 
$$\delta_i(y_j) =  \begin{cases}   1 &\text{if} \ i = j, \\ 0 &\text{otherwise},   \end{cases} \quad \text{for} \ 1 \leq i, j \leq d. $$
(In particular $\delta_1,\ldots, \delta_d$ generate $\Der_\ell(R_\n)$ as a $R_\n$-module).\\

We note that $(\Der_\ell(R))_\n \subseteq \Der_\ell(R_\n) $.  By Claim 1: we get $\rank \Der_\ell(R_\n) \geq d$ as a $R_\n$-module. Using \cite[Theorem 30.7]{Mat} and Lemma \ref{field-extn}  we get that $\rank \Der_\ell(R_\n) \leq d$ as a $R_\n$-module. So $\rank \Der_\ell(R_\n)  = d$. 
Set $z_i = $ image of $y_i$ in $A_\q$.
As $R_\n = A_\q$ is regular local and $z_1,\ldots, z_d$ is a regular system of parameters of $A_\q$, by \cite[Theorem 30.6]{Mat} we get that 
There also exists $\delta_i \in \Der_\ell(R_\n)$ such that 
\begin{equation}\label{tgt}
\delta_i(z_j) =  \begin{cases}   1 &\text{if} \ i = j, \\ 0 &\text{otherwise},   \end{cases} \quad \text{for} \ 1 \leq i, j \leq d.
\end{equation}
The result follows since as  $A$ is a domain we get $A \subseteq A_\q$.

By an argument similar to that in Claim 1 we get $\Der_T(A)_\q \subseteq  \Der_\ell(A)_\q$. More is true. In fact we\\
\textit{Claim-3:} For $i = 1, \ldots, d$ there exists $D_i \in \Der_T(A) $ such that $ \delta_i =  D_i/s_i$ for some $s_i\notin \q$. In particular $\Der_T(A)_\q = \Der_\ell(A_\q)$.    \\

We prove it for $i = 1$. The argument for $i > 1$ is similar. 
Set $W = T[y_1,\ldots,y_d]$.  As $\delta_1$ is $T$-linear and $\delta_1(y_j) = 1$ if $j = 1$ and $0$ if $j > 1$ we get that restricted map  $(\delta_1)_W \in \Der_T(W) $. We note that $W \cong T[Y_1,\ldots, Y_d]$ and $(\delta_1)_W$ is usual differentiation \wrt \ $Y_1$.

We now \\
\textit{Claim-4:} $\delta_1(V) \subseteq V$.  \\
Assume the claim for the moment. Now $A$ is finitely generated as a $V$-module. Say $A = Va_1 + \cdots + Va_c$. Say  $\delta_1(a_j) = u_j/t_j$ where $a_j, t_j \in A$ and $t_j \notin \q$. Set $D_1 = s_1\delta_1$ where $s_1 = t_1\cdots t_c$. Notice $D_1(V) \subseteq A$. Also $D_1 \in \Der_\ell(A_\q)$. 
 It is clear that $D_1(A) \subseteq A$ and $D_1$ is $T$-linear. Thus $(D_1)_A \in \Der_T(A)$ and so $D_1 = s_1 \delta_1 $ in $\Der_\ell(A_\q)$. The result follows.
 
 We now give a proof of Claim 4: \\

Set $K = R/\n = \kappa(\q)$ the residue field of $A_\q$. Note that the $\q A_\q$ completion of $A_\q$ is $\widehat{A_\q} = K [[z_1,\ldots, z_d]]$. (Recall  that $z_i$ is the image of $y_i$ in $A_\q$). Furthermore $\delta_1$ extents to a $K$-linear derivation on $\widehat{A_\q}$ and it is in fact differentiation \wrt \  $z_1$. 

We now note that we have an inclusion of rings $V = T[[y_1,\ldots, y_d]]  \subseteq \widehat{A_\q}$. Furthermore $V$ is complete \wrt \ $I = (y_1,\ldots, y_d)$ and $I \widehat{A_\q} = \q \widehat{A_\q}$. Let $\xi \in V$. Write 
\[
\xi = \sum_{ j \geq 0} t_j y_1^j \quad \text{where} \ t_j \in T[[y_2,\cdots, y_d]].
\] 
Set 
\[
\xi_n = \sum_{j = 0}^{n} t_jy_1^j.
\]
Notice $\xi_n \in W$ and $\xi_n \rt \xi $ in $V$ (\wrt \ the $I$-adic topology on $V$).
Set
\[
\eta = \sum_{j\geq 1} jt_j y_1^{j-1} \quad \text{and} \  \eta_n = \sum_{j = 1}^{n}jt_jy_1^{j-1}.
\]
We note that $\eta_n \rt \eta $ in $V$. 

By \ref{f2} we get that $\xi_n \rt \xi $ in $\widehat{A_\q}$. As $\delta_1$ is continuous \wrt \ $\q\widehat{A_\q}$-adic topology in $\widehat{A_\q}$ we get that
$\delta_1(\xi_n) \rt \delta_1(\xi). $ Notice $\delta_1(\xi_n) = \eta_n$. It follows (using \ref{f2}) that $\delta_1(\xi) = \eta.$ Thus $\delta_1(V) \subseteq V$ and we have proved Claim 4.

(2) We have an inclusion of $R$-modules $\Der_T(A)_f \subseteq \Der_\ell(R)$. If $\n = \q R$ is a maximal ideal in $R$ (where $\q$ is a height $d$ prime ideal in $A$ and $f \notin \q$) then we have $\Der_\ell(R)_\n \subseteq \Der_\ell(R_\n)$. Note $R_\n = A_\q$ and by Claim 3 we have that $\Der_T(A)_\q = \Der_\ell(A_\q)$. In particular we have
$(\Der_T(A)_f)_\n = (\Der_\ell(R))_\n$ for every maximal ideal $\n$ of $R$. Therefore $\Der_T(A)_f = \Der_\ell(R)$. 

(3) (a). This follows from (2) and Claim 3. \\
(3)(b). This follows from Claim 2 and 3(a). \\

(4). This follows from (3).
\end{proof}

\section{Proof of Theorem \ref{main}}
In this section we prove our main Theorem. Let us first recall a result from the influential book \cite{B}.

\s \label{B-Hyp} Let $K$ be a field of characteristic zero  and let $R$ be a commutative Noetherian domain containing $K$ as a subring. Let $\g$ be the set of $K$-linear derivations of $R$ and let $D(R)$ be the subring of $\Hom_K(R,R)$ generated by $\g$ and the multiplication operators defined by elements in the ring $R$.

Let $\m$ be a maximal ideal of $R$. If $\delta \in \g$ then $\delta(\m^2) \subseteq \m$ and so $\delta$ induces a $R/\m$-linear map from $\m/\m^2$ to $R/\m$ which is called the \emph{tangent map} of $\delta$ at $\m$. We say that $\g$ has \textit{maximal rank} at $\m$ if every $R/\m$ linear map from $\m/\m^2$ to $R/\m$ is the tangent map of some $\delta \in \g$.

Now consider the following conditions:
\begin{enumerate}
\item
$\g$ has maximal rank at every maximal ideal in $R$.
\item
There exists an integer $n$ such that $\dim_{R/\m} \m/\m^2 \leq n$ for all maximal ideals $\m$ of $R$ and equality holds for some $\m$.
\item
The residue fields $R/\m$ are algebraic over $K$ for all maximal ideals $\m$ of $R$.
\item
If $M$ is a $R$-module and if $M_\m = M \otimes_R R_\m$ is finitely generated as a $R_\m$-module for all maximal ideals $\m$ of $R$ then $M$ is finitely generated $R$-module.
\end{enumerate}
Then the following is \cite[Chapter 2, Theorem 1.2]{B}:
\begin{theorem}\label{B-Th}[ with hypotheses as in \ref{B-Hyp}]
If the conditions (1)-(4) hold then $D(R)$ is left and right Noetherian and global dimension of $D(R)$ is $n$.

\begin{remark}\label{sutli}
Only condition (4) above is difficult to verify. However it is used only to prove $\g$ is finitely generated as a $R$-module. Thus if we can independently verify that $\g$ is a finitely generated $R$-module then in Theorem \ref{B-Th} all we need is conditions (1)-(3).
\end{remark}

We now give
\begin{proof}[Proof of Theorem \ref{main}.]
By Theorem \ref{chief} we get that $\g = \Der_\ell(R)$ is finitely generated as a $R$-module. Thus by Remark \ref{sutli} we only need to verify conditions (1)-(3) above.

By Proposition \ref{max} we get that $\dim_{R/\n} \n/\n^2 = d$ for each maximal ideal of $\n$ of $R$. Also by Lemma \ref{field-extn}  we get that $R/\n$ is a finite extension of $\ell$ for each  maximal ideal $\n$ of $R$. Thus we have verified conditions (2) and (3).

Let $\n$ be a maximal ideal of $R$. Then by Theorem \ref{chief} we get that $\Der_\ell(R)_\n = \Der_\ell(R_\n)$. Again by Theorem \ref{chief} we get that $\Der_\ell(R_\n)$ is a free $R_\n$-module of rank $d$. Furthermore by (\ref{tgt}) there exists a regular system of parameters $z_1,\ldots, z_d$ of $R_\n$ and $\delta_1,\ldots, \delta_d \in \Der_\ell(R_\n)$ such that
 \begin{equation}\label{tgt-2}
\delta_i(z_j) =  \begin{cases}   1 &\text{if} \ i = j, \\ 0 &\text{otherwise},   \end{cases} \quad \text{for} \ 1 \leq i, j \leq d.
\end{equation} 
We note that $z_1+ \n^2, \ldots, z_d + \n^2$ is a basis of the $R/\n$-vector space $\n/\n^2$. 

As $\Der_\ell(R)_\n = \Der_\ell(R_\n)$ there exists $D_i \in \Der_\ell(R)$ and $s_i \notin{\m}$ such that $\delta_i = D_i/s_i$ for $i = 1,\ldots, d$. Let $\ov{s_i}$ denote the image of $s_i$ in $R/\n$. Note $\ov{s_i} \neq 0$ for all $i$.
Thus by (\ref{tgt-2}) we get
 \begin{equation}\label{tgt-3}
D_i(z_j) =  \begin{cases}   \ov{s_i} &\text{if} \ i = j, \\ 0 &\text{otherwise},   \end{cases} \quad \text{for} \ 1 \leq i, j \leq d.
\end{equation} 
It is now elementary to show that $\g$ has maximal rank at $\n$. Thus we have verified (1). The result now follows from Theorem \ref{B-Th}.
\end{proof}
\end{theorem}
We now ask:
\begin{question}
(with hypotheses as in Theorem \ref{main}). Is $D(R)$ a domain?
\end{question}
\section{Lyubeznik's conjecture for rings considered in this paper}
Let $R$ be the regular domain as in \ref{Hypothesis} and let $D(R)$ be the ring of $\ell$-linear differential operators on $R$. We first show that if Question \ref{Ly} has an affirmative answer then Lyubeznik's conjecture holds for $R$. We then show that a positive answer to a question on Bernstein Sato polynomials of power series will enable us to solve Question \ref{Ly}.

\s We note that $R$ can be considered both as a subring of $D(R)$ and also as a $D(R)$-module. Furthermore it is clear that $R$ is finitely generated as a $D(R)$-module.
Let $h \in R$ with $h \neq 0$. The usual arguments yield that $R_h$ is a $D(R)$-module. We ask
\begin{question}\label{Ly}
Is $R_h$  finitely generated as a $D(R)$-module?
\end{question}
We now show:
\begin{theorem}\label{LC}
If Question \ref{Ly} has an affirmative answer then Lyubeznik's conjecture hold's for $R$, i.e., if $I$ is any ideal in $R$ then $\Ass_R H^i_I(R)$ is finite for any ideal $I$ in $R$ and for $i \geq 0$.
\end{theorem}
For this we need the following two results:
\begin{lemma}\label{Lem-L-1}
Let $I$ be an ideal in $R$. Then for each $i \geq 0$ the local cohomology module $H^i_i(R)$ is a $D(R)$-module. Furthermore if Question \ref{Ly} has an affirmative question then $H^i_i(R)$ is finitely generated as a $D(R)$-module for all $i \geq 0$.
\end{lemma}
\begin{proof}
Let $I = (h_1,\ldots, h_s)$ and consider the \emph{modified \v{C}ech complex} $\C$ on $h_1,\ldots, h_s$. Note that $\C_i$ is a finite direct sum of modules $R_{h_{j_1}\cdots h_{j_i}}$ and the maps are the natural ones upto a sign. It follows that $\C$
is in fact a complex of (left)-$D(R)$-modules. It follows that $H^i_I(R) = H^i(\C)$ is a $D(R)$-module.

If Question \ref{Ly} has an affirmative answer then $\C_i$ is finitely generated $D(R)$-module for each $i \geq 0$.  As $D(R)$ is left Noetherian it follows that $H^i(\C) = H^i_I(R)$ is finitely generated as a $D(R)$-module.
\end{proof}

We also need the following result:
\begin{lemma}\label{Lem-L-2}
Let $M$ be a finitely generated $D(R)$-module. Then $\Ass_R M$ is a finite set.
\end{lemma} 
\begin{proof}
Let $I$ be an ideal in $R$. Set
$$\Gamma_{I}(M) = \{ m \in M \mid I^nm = 0, \ \text{for some} \ n \geq 1 \} = H^0_{I}(M). $$
Then $\Gamma_I(M)$ is a $D(R)$-submodule of $M$. 

The following is a standard argument for proving finiteness of associate primes of modules which are Noetherian over a ring of Differential operators, for instance see \cite{Lyu-1}.
We give it  here for the convenience of the reader.

We claim there is a finite filtration of $M$ by $D(R)$-submodules
$ 0  \subseteq M_1 \subseteq M_2 \cdots \subseteq M_{s-1} \subseteq M_s = M$ such that $M_j/M_{j-1}$ has only one associated
prime for $j = 1,\ldots, s$. For let $P_1$ be a maximal element in the set of the associated primes of $M$.
Then $\Gamma_{P_1} (M) $ is non-zero and has only one associated prime, namely, $ P_1$.  Set
$M_1 = \Gamma_{P_1} (M)$.  As argued before $M_1$  is a $D(R)$-submodule of  $M$, so $M/M_1$ is a $D(R)$-
module. Let $P_2 $ be a maximal element in the set of the associated primes of  $M/M_1$.
Then $\Gamma_{P_2}(M/M_1)$ is a non-zero $D(R)$-submodule of $M/M_1$ and has only one
associated prime, namely, $ P_2$. Set $M_2$ to be the preimage of  $\Gamma_{P_2} (M/M_1)$  in $M$. Since
$M$ is Noetherian  this process eventually stops. This proves the claim.
The set of the associated primes of M is contained in the union of the sets of the
associated primes of all  $M_i/M_{i-1}$  where $i  = 1, \ldots, s$. This proves our Lemma.
 
\end{proof}

We now give:
\begin{proof}[ Proof of Theorem \ref{LC}]
This follows from Lemma's \ref{Lem-L-1} and \ref{Lem-L-2}.
\end{proof}

\s \label{bs-s} \emph{A relation between Question \ref{Ly}  and Bernstein-Sato polynomial.}
Let $K$ be a field of characteristic zero and let $S = K[[X_1,\ldots, X_n]]$. Let $D(S)$ be the ring of $K$-linear differential operators over $S$.
Let $h \in S$ be a non-unit and let $b_h(z)$ be it's Bernstein-Sato polynomial (see \cite[Chapter 3, Corollary 3.6]{B} and also  \cite[Chapter 1, Remark 5.8]{B}).
Let $-c \in \ZZ$  is a lower bound  negative integer root of $b_h(z)$. Then it is easy to verify  that $S_h$ is generated as a $D(S)$-module by $h^{-c}$, see \cite[p.\ 460]{ABL}.
If $h \in S$ is in fact a polynomial then $-n $ is a lower  bound for the  roots of $b_h(z)$, see \cite{V}. Set $bs(h)$ to be the smallest negative  integer root of $b_h(z)$.

Set $\m = (X_1,\ldots, X_n)$ and if $h \in S$ is non-zero then set
\[
v(h) = \max \{ r \mid h \in \m^r \}.
\] 
This is a non-negative integer,  since by Krull's intersection theorem we have $\cap_{r \geq 1} \m^r = 0$. 
We now state our next:
\begin{question}\label{bs}
[with hypotheses as in \ref{bs-s}.]  Let $m$ be a positive integer. Is 
\[
K(m) = \sup\{ bs(h) \mid v(h) \leq m , \text{where}  \ h \in S\} \  \text{finite?}
\]
\end{question}
I believe that  to  answer this question it suffices to consider the case $K = \mathbb{C}$, the complex numbers. Motivated by this we make our final 
\begin{question}\label{bs-K} [with hypotheses as in Question \ref{bs}.]
Does there exists $c > 0$ such that $K(m) \leq c$ for any field $K$ of characteristic zero?
\end{question}
We now state the main result of this section
\begin{theorem}\label{rel}
If Question \ref{bs-K} has an affirmative answer for all  positive integers $m$ then so does Question \ref{Ly}. 
\end{theorem}
\begin{proof}
Let $\n$ be a maximal ideal of $R$. As $R$ is a domain we have $\cap_{i \geq 1}\n^i = 0$, see \cite[ 8.10(ii)]{Mat}. Let $h \in R$ be non-zero.
Then
\[
v_\n(h) = \max\{i \mid h \in \n^i \} \quad \text{is a non-negative integer}.
\]
We now
\textit{Claim-1:} The set $V(h) =  \{ v_\n(h) \mid \n \ \text{a maximal ideal of } \ R \}$  is bounded above. \\
Proof of Claim-1:   Suppose if possible Claim-1 is not true. Then  for any positive integer $i$ 
there exists a maximal ideal $\n_i$ with $v_{\n_i}(h) \geq i$.

As $R$ is the localization of a complete local ring it is excellent. So by \cite{DC}  there exists $c > 0$ such that
\[
\n^j \cap (h) = \n^{j-c}((\n^c \cap (h)) \quad \text{for all} \ j \geq c \ \text{and for all maximal ideals $\n$ of $R$}.
\]
Choose $\n_i$ with $v_{\n_i}(h) > c + 2$.
Then we have 
$$ (h) = \n_i^{c+1} \cap (t)  = \n_i(\n_i^c \cap (h))  = \n_i (h).$$
It follows that there exists $\xi \in \n_i$ with $(1-\xi)(h) = 0$. As $R$ is a domain and $h \neq 0$ this forces $1-\xi = 0$ and so $\xi = 1 \in \n_i$. This is a contradiction. Thus Claim-1 is true.  

Let $m$ be an upper bound for $V(h)$.  As we are  assuming that Question \ref{bs-K} has an affirmative answer we get that  there exists  $c >0$  such that 
  $K(m) \leq c$ for any field $K$ of characteristic zero. 
  
  Consider the the following ascending chain $\FF$ of $D(R)$-submodules of $R_h$ whose union is $R_h$
  \[
  D(R)\frac{1}{h} \subseteq D(R)\frac{1}{h^2} \subseteq \cdots \subseteq D(R)\frac{1}{h^{c}} \subseteq \cdots \subseteq D(R)\frac{1}{h^p} \subseteq \cdots
  \]  
  We say $\FF$ stablizes  at level $q$ if $D(R)h^{-p} = D(R)h^{-q}$ for all $p \geq  q$.
  
   Let $\n$ be a maximal ideal of $R$. We  localize $\FF$  at $\n$ to get the ascending chain $\FF_\n$. We then tensor it with $\widehat{R_\n}$,  the completion of $R_\n$,  to get the ascending chain $\widehat{\FF_\n}$.  Set $\kappa(\n) = R/\n$.  Note  $\widehat{R_\n} \cong  \kappa(\n)[[Z_1,\ldots, Z_d]]$. Let $D(\widehat{R_\n})$ be the ring of $\kappa(\n)$-linear differential operators on $\widehat{R_\n}$. Then by \cite[Chapter 3, Lemma 1.5]{B} we get that
\[
D(R)\otimes_R \widehat{R_\n} \cong D(\widehat{R_\n})
\]
and we also get that $\widehat{\FF_\n}$ is an ascending chain of $D(\widehat{R_\n})$-submodules of $(\widehat{R_\n})_h$.  It follows from \ref{bs-s} that $\widehat{\FF_\n}$ stablizes at  level $c$. As the map $R_\n \rt \widehat{R_\n}$ is faithfully flat we get that $\FF_\n$ stablizes at  level $c$. 

We have shown that   $\FF_\n$ stablizes at level  $c$ for any maximal ideal $\n$ of $R$. It follows that   $\FF$ stablizes at level  $c$. Therefore $R_h$ is generated as a $D(R)$-module by $1/h^{c}$. In particular $R_h$ is finitely generated as a $D(R)$-module.
\end{proof}

\section{examples}
In this section we show that for each $d \geq 1$ there exist infinitely  many examples of regular  rings which satisfy our hypothesis \ref{Hypothesis}. For simplicity we will assume that $k$ is an algebraically closed field of characteristic zero. 

\begin{example}\label{Veronese}
Let $Q = k[x_1,\cdots, x_d, x_{d+1}]$  where $d \geq 2$. Set \\ $S = \widehat{Q} = k[[x_1,\cdots, x_d, x_{d+1}]]$. Let $n \geq 2$ be a positive integer and let $\zeta$ be a primitive $n^{th}$-root of unity and let $G = < \zeta^i \colon 0 \leq i \leq n-1>$. Then $G$ acts on both $Q$ and $S$ with the action $x_i \mapsto \zeta x_i$.
Let $B = Q^G$ and $A = S^G$. Note that $B \cong Q^{<n>}$ the $n^{th}$ Veronese subring of $Q$ and that $A = \widehat{B}$ the completion of $B$ at it's irrelevant maximal ideal. As $\Proj(B)$ is smooth we get that $A$ is an isolated singularity.
It is well known that $Cl(B)$, the class group of $B$ is $\ZZ/n \ZZ$. As $\Proj(B)$ is smooth and $\dim B = d+ 1 \geq 3$ we get that $B$ satisfies   $R_2$ property of Serre. So by a result of Flenner, $Cl(A) \cong Cl(B)$. 

Let $f = x_1^n + \cdots + x_{d+1}^n$. As $d \geq 2$, it is well-known that $f$ is irreducible in $Q$. Note $f \in A$. Let $\m$ be the maximal ideal of $S$. If $T$ is a quotient ring of $S$ then set $G(T) = \bigoplus_{n \geq 0} \m^n T / \m^{n+1}T$ the  
associated graded ring of $T$ \wrt \ its maximal ideal $\m T$.
Note $G(S/fS) \cong G(S)/f G(S) = Q/f Q$ which is a domain. So $S/fS$ is a domain. In particular $fS$ is a prime ideal in $S$. As $fA = fS \cap A$ we get that $fA$ is a prime ideal in $A$.

Set $R_{n,d} = A_f$. By the localization sequence of class groups we have $Cl(R_{n,d}) = \ZZ/n \ZZ$. Also note that $\dim R_{n,d} = d \geq 2$.
\end{example}

In \ref{Veronese} we had the restriction that $d \geq 2$ and that $R$ is not a UFD.
Next we give infinitely many one dimensional examples satisfying \ref{Hypothesis}. We also give infinitely many examples satisfying \ref{Hypothesis} of dimension $d \geq 3$ which are also UFD's.  We need to recall the notion of simple singularities.

\s \emph{Simple singularities:} Let $S = k[[x,y,z_2, \ldots, z_d]]$ with $d \geq 2$. Simple singularities are defined by the following equations:
\begin{align*}
(A_n) \quad &x^2 + y^{n+1} + \sum_{j = 2}^{d} z_j^2  &(n \geq 1), \\
(D_n) \quad &x^2y + y^{n-1} + \sum_{j = 2}^{d} z_j^2  &(n \geq 4), \\
(E_6) \quad &x^3 + y^{4} + \sum_{j = 2}^{d} z_j^2,  &   \\
(E_7) \quad &x^3 + xy^{3} + \sum_{j = 2}^{d} z_j^2,  &   \\
(E_8) \quad &x^3 + y^{5} + \sum_{j = 2}^{d} z_j^2.  & 
\end{align*}

\s \label{simple} Let $A = Q/(f)$ be a simple singularity. Then $A$ is an isolated singularity. In particular by a result due to Grothendieck $A$ is a UFD if $\dim A \geq 4$. 
We also  note that if $d \geq 2$ then $A/(z_d)$ is a simple singularity of the same type.

\s \label{Groth} \emph{Grothendieck Groups:}
Let $T$ be a commutative Noetherian ring and let $\mode(T)$ denote the category of all finitely generated $T$-modules. Let $\eU$ be an additive subcategory of 
$\mode(T)$ closed under extensions and let $\Gr(\eU)$ denote the \emph{Grothendieck group} of $\eU$. We recall the following  three facts of Grothendieck groups that we need.
\begin{enumerate}
\item
Let $(A,\m)$ be  a  \CM \ local domain. Let $\eC$ be the additive subcategory  of $\mode(A)$ consisting of all maximal \CM \ $A$-modules. Then
\begin{enumerate}
\item
The inclusion $i: \eC \rt \mode(A)$ induces an isomorphism of Grothendieck groups
$\Gr(\eC)$  and $\Gr(\mode(A))$, cf. \cite[13.2]{Y}.
\item
The map
$\rk \colon \Gr(\eC) \rt \ZZ$ defined by $[M] \mapsto \rank(M)$ is well-defined surjective group homomorphism. We have an isomorphism $\ZZ \oplus \ker \rk  \rt \Gr(\eC)$ where $(1,0) \mapsto [A]$.
\end{enumerate}
\item
Let $T$ be a regular ring of finite Krull dimension and let $K(T)$ be its K-group.
Then the natural map $K(T) \rt \Gr(\mode(T))$ is an isomorphism.
\item
Let $f \in T$.
The sequence
\[
\Gr(T/(f)) \xrightarrow{d_1} \Gr(T) \xrightarrow{d_0} \Gr(T_f) \rt 0
\]
is exact. Here
\[
d_1([M]) = [M] \quad \text{and} \quad d_0([N]) = [N_f].
\]
\end{enumerate}
\begin{remark}
If $f$ is $T$-regular then note that the class of $[T/(f)]$ is zero in $\Gr(T)$.
The reason is that we have an exact sequence $0 \rt T \xrightarrow{f} T \rt T/(f) \rt 0$.
\end{remark}

\begin{remark}\label{Gr-simple}
The Grothendieck groups of all simple singularities is known, see \cite[13.10]{Y}.
We will only need the following fact: Let $A$ be an $A_n$ singularity of dimension $l$. Then
\begin{enumerate}
\item
If $n$ is even then $\Gr(A) = \ZZ$ if $l$ is odd and is equal to $\ZZ \oplus \ZZ/(n+1)\ZZ$ if $l$ is even.
\item
If $n$ is odd then $\Gr(A) = \ZZ^2$ if $l$ is odd and is equal to $\ZZ \oplus \ZZ/(n+1)\ZZ$ if $l$ is even.
\end{enumerate}
\end{remark}

\begin{example}
Let $S = k[[x,y,z_2, \ldots, z_d]]$ with $d \geq 2$ and let $A = S/(f)$ be an $A_n$-singularity with $n$ \textit{even}. Note $\dim A = d + 1$. Set $R_{n,d} = A_{z_d}$.
We note that if $\dim A \geq 4$ then $A$ is a UFD and so $R_{n,d}$ is also a UFD.\\
\emph{Case 1:} $\dim A = d + 1$ is \emph{even}.\\
Consider the exact  sequence
\[
\Gr(A/(z_d)) \xrightarrow{d_1} \Gr(A) \xrightarrow{d_0} \Gr(R_{n,d}) \rt 0.
\]
Note $A/(z_d)$ is an $A_n$ singularity of dimension $d$. Also for all $d \geq 1$ the ring $A/(z_d)$ is a domain. By \ref{Groth} and \ref{Gr-simple} we have that
$\Gr(A/(z_d)) = \ZZ$ and is generated by the class of $A/(z_d)$.  
 By \ref{Groth}(3)
it follows that $d_1 = 0$.   It follows that 
\[
\ZZ \oplus \ZZ/(n+1) = \Gr(A) \cong \Gr(R_{n,d}) \cong K(R_{n,d}).
\]

\emph{Case 2:} $\dim A = d + 1$ is \emph{odd}.\\
We again consider the exact sequence
\[
\Gr(A/(z_d)) \xrightarrow{d_1} \Gr(A) \xrightarrow{d_0} \Gr(R_{n,d}) \rt 0.
\]
We again assert that $d_1 = 0$. Notice 
$\Gr(A/(z_d)) = \ZZ \oplus \ZZ/(n+1)$ and $\Gr(A) = \ZZ$. Clearly $d_1(\ZZ/(n+1)) = 0$. By \ref{Groth}(1)(b) the element $(1,0)$ of $\Gr(A/(z_d))$ is generated by the class of $A/(z_d)$. By \ref{Groth}(3) we get $d_1([A/(z_d)]) = 0$. Thus again $d_1 = 0$. So we have
\[
\ZZ  = \Gr(A) \cong \Gr(R_{n,d}) \cong K(R_{n,d}).
\]
\end{example}

\begin{remark}
The point of this section was to show that there exists infinitely many non-isomorphic regular rings satisfying our hypothesis \ref{Hypothesis}. I believe that 
the examples given above is only a tip of the iceberg. There should be many more examples. However I do not know how to prove they are non-isomorphic. 
\end{remark}

\begin{remark}
The reason why the above remark is pertinent is due to a related result which we now describe. Let $S = \bigoplus_{n \geq 0}S_n$ be a standard graded algebra over  an \textit{uncountable} algebraically closed field $k = S_0$ with $\Proj(S)$ smooth. Then there exists an uncountable family $\{f_\alpha \mid \alpha \in \Gamma \}$ of homogeneous elements  of positive degree with $S_\alpha \ncong S_\beta$ for $\alpha, \beta \in \Gamma$ and $f_\alpha \neq f_\beta$. This follows from a result in \cite{SR}. 
\end{remark}

\end{document}